\documentclass[default]{sn-jnl}

\usepackage{amsmath,graphicx}
\usepackage{matlab-prettifier}
\usepackage{float}
\usepackage{mdframed}
\usepackage[linesnumbered,ruled,vlined]{algorithm2e}

\usepackage{color,verbatim}
\usepackage{amsmath,amsthm}
\usepackage{amssymb,amsfonts}
\usepackage{hyperref}
\usepackage{times}
\usepackage{epstopdf}
\usepackage{subcaption}
\usepackage{listings}%
\usepackage{graphicx}%
\usepackage{multirow}%
\usepackage{mathrsfs}%
\usepackage[title]{appendix}%
\usepackage{xcolor}%
\usepackage{textcomp}%
\usepackage{manyfoot}%
\usepackage{booktabs}%
\usepackage{multirow}
\usepackage{algorithm2e}

\newtheorem{theorem}{Theorem}

\newtheorem{lemma}[theorem]{Lemma}
\newtheorem{remark}[theorem]{Remark}

\newcommand{\rset}{\mathbf{R}}

\newcommand{\BigO}{\mathcal{O}}

\newcommand{\diag}{\text{diag}}

\providecommand{\norm}[1]{\lVert#1\rVert}

\newcommand{\be}{\begin{equation}}
\newcommand{\ee}{\end{equation}}
\newcommand{\bt}{\begin{tabular}}
\newcommand{\et}{\end{tabular}}
\newcommand{\bc}{\begin{center}}
\newcommand{\ec}{\end{center}}

\begin{document}

\title[Learning Explicitly Conditioned Sparsifying Transforms ]{Learning Explicitly Conditioned Sparsifying Transforms }

\author*{\fnm{Andrei} \sur{Pătrașcu}}\email{andrei.patrascu@fmi.unibuc.ro}

\author{\fnm{Cristian} \sur{Rusu}}\email{cristian.rusu@fmi.unibuc.ro}

\author{\fnm{Paul} \sur{Irofti}}\email{paul@irofti.net}

\affil{\orgdiv{Research Center for Logic, Optimization and Security (LOS), Department of Computer Science, Faculty of Mathematics and Computer Science}, \orgname{University of Bucharest}, \orgaddress{\street{Academiei 14}, \city{Bucharest}, \postcode{010014}, \state{Romania} }}

\abstract{ 
Sparsifying transforms became in the last decades widely known tools for finding structured sparse representations of signals in certain transform domains.
Despite the popularity of classical transforms such as DCT and Wavelet, learning optimal transforms that guarantee good representations of data into the sparse domain has been recently analyzed in a series of papers.
Typically, the conditioning number and representation ability are complementary key features of learning square transforms that may not be explicitly controlled in a given optimization model. Unlike the existing approaches from the literature, in our paper, we consider a new sparsifying transform model that enforces explicit control over the data representation quality and the condition number of the learned transforms. We confirm through numerical experiments that our model presents better numerical behavior than the state-of-the-art.   
}

\keywords{conditioning number, alternating minimization, transform learning, sparse coding}

\maketitle

\section{Introduction}

In the context of signal processing and machine learning, there are many data-driven numerical algorithms that construct linear transformations satisfying desirable properties. In the past 15 years, chief among these properties is the ability to sparsely represent data, especially image data. These sparsifying transformations have multiple applications such as image denoising, inpainting, compressed sensing and structured dictionary learning \cite{EladAharon06_denoising,RavBre:12,Sha:21,RD13_SBO,dl_book}.

A major technical issue related to these learnt transformations is their numerical conditioning. For example, methods that use these transformations in image processing are numerically stable only when they enjoy good numerical properties, i.e., low (or at least bounded and controllable) condition numbers.
Many statistical models such as the linear (quadratic) discriminant analysis and principal component analysis, use the covariance matrix as a fundamental tool for essential information extraction from high-dimensional data. However, often when the data dimension exceeds the number of observations the sample covariance matrix becomes poorly conditioned or singular even in simple cases \cite{Wan:21}. One remedy is, instead of using the original sample covariance, to use the nearest positive-definite approximation of the original matrix, which has been commonly considered in many contexts (see \cite{Li:20, Wan:21} and the references therein). This approximation formulates as the solution of a Frobenius norm projection onto the set of well-conditioned positive-definite matrices, which proved to allow efficient computational schemes based on singular value decomposition (SVD). Thus, in \cite{Li:20}, the main matrix projection operation was reduced to a simple vector projection problem onto an intersection between box constraints and linear hyperplanes. Although a closed-form solution could not be derived, an optimal linear time algorithm was presented.

A more general formulation of the previously mentioned projection problem was considered in \cite{Ada13}, where a Procrustes objective function is minimized over the same well-conditioned positive-definite matrices set. In order to avoid projection operation, the author introduces an alternating minimization scheme by explicitly optimizing the matrix components of the SVD factorization. Other higher generalizations of the Procrustes problem have been considered in \cite{Ful:23,Chu:01,Chu:98,Fra:14}. In many of these cases, including the original Procrustes problem, the solutions are based on polar decompositions of some matrix products.



Among the first sparsifying transform learning techniques we mention \cite{RavBre:12,RavWen:15alg,RavWen:15conv}, well-conditioned transforms are computed through direct penalization costs. The penalty strategy maintains a trade-off between the conditioning and representation quality using additional penalty functions to control the condition number of the optimal transform. Often, this optimal conditioning is invisibly dependent on the penalty parameters, which may ask for fine-tuning in some cases. We will bring more details about this approach in the next section. Since its introduction in \cite{RavBre:12}, the approach has gained popularity and has been extended to include: convergence analysis \cite{RavWen:15conv, RB15_SparseTransoformDenoising}, an online approach for streaming large datasets \cite{RavWen:15alg, RavWen:15conv}, a double-sparse approach for image processing \cite{RB13_DoubleSparseImages} etc. To further boost the performance of these methods, recent work was done to extend sparsifying transforms to include kernel methods \cite{MAGGU2023141}. The work in \cite{Sha:21} uses the conceptual idea of transform learning to bridge the gap between state-space models and neural network approaches by introducing sequential transform learning techniques. In \cite{Kumar23_TLSubspaceInterpolation}
the authors employ multiple transform learning steps,
where at each step
a new transform is learned together with its sparse representations
through the closed-form solutions proposed by \cite{RavBre:12}.
The input signals at each step
are the residuals from the previous step
and
the process is applied until an error threshold is reached. BLORC~\cite{Ghosh24_TLBilevelOptimization} represents an alternative
to the closed-form transform learning approach
which employs an online gradient-descent based method
where the signals are processed in mini-batches
throughout multiple epochs.


\vspace{0.5cm}

\noindent \textbf{Contribution.} In this paper, we propose a sparsifying transform learning method that, unlike previously proposed methods from the literature, explicitly controls the condition number of the learnt transform at every step of the algorithm.
Furthermore,
the proposed method is numerically efficient,
fast,
has no hyper-parameters to choose or tune
and, as we will show experimentally, provides state-of-the-art results on image data.

\vspace{10pt}

\noindent \textbf{Notation.}
Let $a_{ij}$ be the element, $a_i$ the $i$-th column and $a^j$ the $j$-th line of matrix $A$.
We denote the scalar Euclidean product $\langle x, y \rangle := x^Ty$, $\norm{A}_F=\sqrt{\sum_{i,j} a_{ij}^2}$ the Frobenius norm,
with $\norm{a}_2$ the $\ell_2$ vector norm,
and with $\ell_0$ the pseudo-norm $\norm{a}_{0}$
counting the number of nonzero elements.
We also denote $[x]_+ = \max\{0, x\}$.
We define $\odot$ as the elementwise vector multiplication
$a \odot b = \{c \mid a_i \times b_i \;\; \forall i\}$
and $\oslash$ is the elementwise vector division 
$a \oslash b = \{c \mid a_i / b_i \;\; \forall i \;\; \text{where} \;\; b_i \neq 0\}$.
The set of $n$-dimensional orthogonal matrices is denoted by $O_n$. For a given matrix $A$ we define the condition number of $A$ as $\kappa(A) = \frac{\sigma_{\max}}{\sigma_{\min}}$, where $\sigma_{\max}$ and $\sigma_{\min}$ are the largest and smallest, respectively, singular values of $A$.

\noindent \textbf{Paper structure.} 
We proceed in Section~\ref{sec:transforms}
to introduce a new formulation for well-conditioned transforms that
explicitly constrains the condition number of the transform we learn.
To our knowledge, this is proposed and solved for the first time in the literature.
In Section \ref{sec:algorithm} we describe an alternating optimization procedure to compute well-conditioned transformations for sparse representations.
Section~\ref{sec:spectrum_projection} analyzes and describes in detail the efficient numerical optimization step that updates the spectrum subjected to the explicit condition number constraint. In Section~\ref{sec:experiments}, we give experimental numerical evidence, on both synthetic and real datasets, to support the effectiveness of the method and to showcase how well it compares with previous state-of-the-art algorithms from the literature.
We conclude the paper in Section~\ref{sec:conclusion}.

\section{Well-conditioned transforms}
\label{sec:transforms}

Let the data $Y \in \rset^{n \times m}$, the sparsifying Transform Learning problem \cite{RavBre:12} seeks the optimal orthogonal transform that maps the data near to a set of sparse vectors. For this purpose, the following minimization problem has to be solved:
\begin{align}\label{ST_PD_model}
\min\limits_{W \in \rset^{n \times n}, X \in \rset^{n 
\times m} } \;\;&  \norm{X - WY}^2_F \\  
\text{subject to\quad\quad} & \; W^TW = I_n, \norm{X_i}_0 \le s \quad \forall i, \nonumber
\end{align}
where $W$ is the orthogonal matrix that represents $Y$ close to  $X$. Also, $s$ is the chosen sparsity level. On short, the solution pair $(X^*,W^*)$ aims to attain a minimal residual $\norm{X^* - W^* Y}$ over the orthogonal transforms $W$ and sparse mappings $X$. Often there is no orthogonal $W^*$ that guarantees a perfect representation $W^*Y = X^*$ or a low residual. Therefore, accurate representations are motivated in several applications such as compression, denoising and hyperspectral imaging \cite{EladAharon06_denoising,RavBre:12,Sha:21}. A relaxation of the orthogonality constraints for this purpose seems the most natural way to obtain lower factorization errors. At the same time, an ill-conditioned transform burdens the data reconstruction from the sparse codes. Overall, there is a trade-off between the two complementary aspects: \textit{sparsifying error} and \textit{transform conditioning}. It has been approached in various ways in the literature.

In the series of papers \cite{RavBre:12,RavWen:15alg,RavWen:15conv}, the authors use a specific barrier function in order to remove the orthogonality constraints, which yields the following model:
\begin{align}\label{Bressler_problem}
\min\limits_{W \in \rset^{n \times n}, X \in \rset^{n 
\times m} } \;\;&  \norm{X - WY}^2_F - \mu \log|\det{W}| + \frac{\rho}{2}\norm{W}^2_F \\ 
\text{subject to\quad\quad} & \; \norm{X_i}_0 \le s \quad \forall i. \nonumber
\end{align}
Although $\log\det(\cdot)$ is a well-known barrier function used to preserve feasibility in the positive semidefinite matrix cone, here it is used to keep the singular values of $W$ away from $0$. At the same time, the second penalty is used to bound from above the squared singular values of the transform. By variation of the penalty parameters $\mu$ and $\rho$ one gains an indirect and ``blind" control over the condition number of the optimal transform $W$. In \cite[Proposition 1]{RavBre:12} we are provided with an upper bound on the conditioning number $\kappa$ of the optimal transform $W^*$ as: 
\begin{align*}
\kappa \le e^{r} + \sqrt{e^{r} - 1} \;\; \text{where} \;\; r = -\log\det{(W^*)} + c\norm{W^*}^2_F - \frac{n}{2} - \frac{n}{2}\log(2c).
\end{align*}
Under specific conditions, it can be proved that the solution $W^*$ may result orthogonal, but in general, it is not clear how to use the previous bound for finding particular parameter values that guarantee explicit conditioning. In our model, we search for the best transform of explicit fixed conditioning that minimizes the representation error.

For this purpose, we aim to fully and explicitly control the conditioning number $\kappa(W)$ using a natural parameterized relaxation of the orthogonality constraints in \eqref{ST_PD_model} as:  let parameter $\rho \ge 1$
\begin{align} \label{approx_relaxed_model}
\min\limits_{W \in \rset^{n \times n}, X \in \rset^{n 
\times m} } \;\;&  \norm{X - WY}^2_F \\
\text{subject to\quad\quad} & \;\; \kappa(W)\le  \rho, \; \norm{X_i}_0 \le s \quad \forall i. \nonumber
\end{align}
On one hand, when $\rho = 1$ the new feasible set from \eqref{approx_relaxed_model} reduces to the original one of \eqref{ST_PD_model}. On the other,  a sufficiently large $\rho$ enlarges the feasible set to the extent that it includes the orthogonal solution of \eqref{ST_PD_model}, the optimal transform of \eqref{Bressler_problem} and even the unconstrained least squares solution. In particular, let $X$ be fixed and denote $W_\text{LS} = \underset{W}{\arg\min} \norm{WY-X}^2_F$, by allowing $\rho \ge \kappa(W_\text{LS}) $ the optimal transform solution of \eqref{approx_relaxed_model} is equal to $W_\text{LS}$. 
Intuitively, the optimal transform that globally solves \eqref{approx_relaxed_model} may attain lower residual values than the penalty-based models, while simultaneously keeping the desired direct conditioning restrictions. In this sense, the parameter $\rho$ gives some direct control over the trade-off between the sparsifying power and the conditioning of the optimal transform. 

A drawback of \eqref{ST_PD_model} consists of having a trivial global solution in $(0,0)$, i.e., both $W$ and $X$ are zero matrices of appropriate sizes. In order to avoid trivial local minima, the authors of \cite{RavBre:12} used a logarithmic barrier that keeps the spectrum of $W$ far from the origin. Also, to measure the trade-off between representation quality and norm of the transform spectrum, they use the ratio $\frac{\norm{X - WY}_F}{\norm{W}_F}$ as a performance indicator in the numerical tests. Motivated by these same issues, we impose an additional parameterized constraint that keeps the norm $\norm{W}_F$ constant, which leads to: 
\begin{align} \label{approx_relaxed_constnorm_model}
\min\limits_{W \in \rset^{n \times n}, X \in \rset^{n 
\times m} } \;\;&  \norm{X - WY}^2_F \\
\text{subject to\quad\quad} & \;\; \kappa(W)\le  \rho, \; \norm{W}_F = \tau, \nonumber\\
& \;\; \norm{X_i}_0 \le s \quad \forall i  \nonumber. 
\end{align}

Note that each new constraint in our proposed problem is equivalent to a regularization term in \eqref{Bressler_problem}. The highly nonconvex nature makes the model \eqref{approx_relaxed_constnorm_model} hard to solve globally in both variables $(W,X)$ simultaneously. However one can easily observe that $\arg\min_{z: \norm{z}_0 \le s} \; \norm{z-x}^2_2$ is the sparse vector that contains on nonzero entries the largest (in magnitude) $s$ components of $x$. Based on this fact, the projection subproblem in the variable $X$ has an explicit solution. 
 


 Although $\kappa(W)\le  \rho$ seems a hard nonsmooth nonconvex constraint, we further provide equivalent formulations based on SVD decomposition that leads to more optimistic viewpoints. 
 Like most previous works have done \cite{RavBre:12,RavWen:15alg,RavWen:15conv}, we develop Alternating Minimization schemes that could eliminate some of the constraints and difficulties. Since the condition number constraints impose limitations only over the spectrum of $W$ we naturally change our decision variable $W$ into the SVD triplet $(U,\Sigma,V)$ and obtain:
\begin{align*}
\min\limits_{U, \Sigma, V \in \rset^{n \times n}, X \in \rset^{n \times m}} &  \; F(U,\Sigma, V, X):=\norm{U\Sigma V^T Y - X}_F^2 \\ 
\text{subject to\quad\quad} \;\; &  U, V \in O_n, \; \kappa(\Sigma) \le \rho, \; \| \Sigma \|_F = \tau, \; \norm{X_i}_0 \le s. 
\end{align*}
Although the residual function $F$ is nonconvex, the new separable constraints are favorable for algorithms that use at each step blockwise minimization over a single matrix variable.  

\section{Proposed alternating minimization algorithm}
\label{sec:algorithm}

The Alternating Minimization (AM) method is a widely known algorithmic technique that, unlike common first-order methods which devise recurrent updates over all dimensions of the search variable, relies on each step choosing a small number of variable coordinates (or blocks) and minimizing a model of the objective function over the chosen block. Common rules for choosing the blocks of coordinates are cyclic, randomized, or greedy. In our case, the cyclic exact AM iteration reduces to: given the current $(U^t, \Sigma^t, V^t, X^t)$, at iteration $t+1$
\begin{align*}
U^{t+1} &:= \underset{U \in O_n}{\arg\min}\ F(U,\Sigma^t,V^t,X^t) \\
\Sigma^{t+1} &:= \underset{\sigma \in \rset^n, \kappa(\diag(\sigma)) \le \rho}{\arg\min}\ F(U^{t+1},\diag(\sigma),V^t,X^t) \\
V^{t+1} &:= \underset{V \in O_n}{\arg\min} \ F(U^{t+1},\Sigma^{t+1},V,X^t) \\
X^{t+1} &:= \underset{\norm{X_i}_0 \le s}{\arg\min} \ F(U^{t+1},\Sigma^{t+1},V^{t+1},X).
\end{align*}

\noindent A common property often shared by many alternating minimization schemes is that the objective function descends and the sequence limit points are stationary. 
\begin{lemma}
The Exact AM~(EAM) is a descent method 
\begin{align*}
F(U^{t+1}, \Sigma^{t+1}, V^{t+1}, X^{t+1}) \le F(U^t, \Sigma^t, V^t, X^t), \forall t \ge 0. 
\end{align*}
\end{lemma}
\begin{proof}
Since $F$ is lower bounded and smooth, the first statement can be easily deduced from the EAM iteration:
\begin{align*}
F(U^{t+1}, \Sigma^{t+1}, V^{t+1}, X^{t+1}) 
& \le F(U^{t+1}, \Sigma^{t+1}, V^{t+1}, X^t) \\
& \le F(U^{t+1}, \Sigma^{t+1}, V^t, X^t) \\
& \le F(U^{t+1}, \Sigma^{t}, V^t, X^t) \\
& \le F(U^{t}, \Sigma^{t}, V^t, X^t), \; \forall t \ge 0,
\end{align*}
\end{proof}
\noindent However, a close inspection of the AM steps reveals some bottlenecks in the implementation of this scheme. We further make a careful analysis of each AM step:

\vspace{5pt}

\noindent $(i)$ Denoting $A^t = \Sigma^t (V^t)^T Y$, minimization over $U$ reduces to a classical Orthogonal Procrustes problem: 
\begin{align*}
\underset{U}{\min}\ \norm{UA^t - X^t}^2_F \;\; \text{subject to} \;\; U^TU = I_m.
\end{align*}

\vspace{5pt}

\noindent $(ii)$ The subproblem over $\Sigma$ reformulates as: 
\begin{align*}
\underset{\Sigma}{\min}\ \norm{\Sigma (V^t)^T Y - (U^{t+1})^TX^t}^2_F \;\; \text{subject to} \;\; \kappa(\Sigma) \le \rho.
\end{align*}
Based on the fact that $\Sigma = \diag(\sigma)$ is diagonal, we ignore for simplicity counter $t$ and derive:
\begin{align*}
 \| U^T X & -  \text{diag}(\mathbf{\sigma})V^T Y \|_F^2 
 = \sum\limits_{i=1}^n \norm{X^Tu_i - Y^Tv_i \sigma_i}^2_2 \\
& = \sum\limits_{i=1}^n \sigma_i^2 \norm{Y^Tv_i }^2_2 - 2 \sigma_i \langle Y^Tv_i, X^Tu_i \rangle + \norm{X^Tu_i}_2^2\\
& \overset{\tilde{\sigma}_i:= \sigma_i \norm{Y^Tv_i }}{=} \sum\limits_{i=1}^n \tilde{\sigma}_i^2 - 2 \tilde{\sigma}_i \frac{\langle Y^Tv_i, X^Tu_i \rangle}{\norm{Y^Tv_i }} + \norm{X^Tu_i}_2^2\\
& = \sum\limits_{i=1}^n \left(\tilde{\sigma}_i^2  - \frac{\langle Y^Tv_i, X^Tu_i \rangle}{\norm{Y^Tv_i }_2} \right)^2 + \norm{X^Tu_i}_2^2 - \left(\frac{\langle Y^Tv_i, X^Tu_i \rangle}{\norm{Y^Tv_i }_2} \right)^2,
\end{align*}
where we have used the change of variables $\tilde{\sigma_i} := \sigma_i \norm{Y^Tv_i }$. Finally, by denoting          $r_{i}  = \norm{Y^Tv_i}_2, D = \text{diag}(r), d_i = \frac{\langle Y^Tv_i, X^Tu_i \rangle}{\norm{Y^Tv_i }_2}$, we shortly conclude that:          
\begin{align}\label{projection_problem}
\underset{\sigma}{\arg\min}\ \norm{U^T X -  \text{diag}(\mathbf{\sigma})V^T Y}_F^2 
            = D^{-1} \cdot \underset{\tilde{\sigma}}{\arg\min}\ \norm{\tilde{\sigma}  - d }^2_2
            \text{ subject to }  \ell r_{i} \le \tilde{\sigma}_i \le r_{i}\kappa \ell.
\end{align}
This final projection problem seems much simpler than the initial form of the $\Sigma$ step. To evaluate its complexity, it will be extensively analyzed in Section \ref{sec:spectrum_projection}.

\vspace{5pt}

\noindent $(iii)$ The minimization over variable $V$ seems the most complicated step. It resembles a Weighted Orthogonal Procrustes problem (or a Penrose Regression problem) that in general does not have a closed-form solution \cite{Chu:01,Chu:98}, i.e. 
\begin{align*}
\min\limits_{V} \ \; \norm{\Sigma^{t+1} V^T Y - (U^{t+1})^TX^t}^2_F \;\; \text{subject to} \;\; V^TV = I_n.
\end{align*}
Previous works \cite{Chu:01,Chu:98,Ful:23} on these topics suggest the need for a specific sub-routine to approximate the optimal $V^{t+1}$, without global optimality guarantees. Our exact cyclic AM scheme requires such a routine to become tractable.   
Thus, instead of using the exact solution of the subproblem, the minimization of a simpler approximation of the objective seems much easier. For example, notice that
\begin{align*}
\norm{\Sigma V^T Y - U^TX}^2_F = \norm{\Sigma (V^T Y - \Sigma^{-1} U^TX)}^2_F \le \sigma_{\max}^2\norm{V^T Y - \Sigma^{-1} U^TX}^2_F.  
\end{align*}
The right-hand side becomes another Procrustes problem in the variable $V$. In the proposed algorithm we minimize this upper approximation, which is tractable.

\begin{algorithm}[h]
\caption{Alternating Minimization (AM) $(Y,\rho,\tau,T)$ }
\label{alg:rip}
		Initialize $\; U^0, \Sigma^0, V^0, X^0, k: = 0$\;
		\For{ $t = 1, \cdots, T$}{
            Use $U^{t+1} := \arg\min \; \norm{U \Sigma^t (V^t)^T Y - X^t}^2_F \;\; \text{s.t.} \;\;  U^TU = I_n$ and solve for $U^{t+1}$ by the closed-form formula for the polar decomposition of matrix $\Sigma^t (V^t)^T Y X^t$\;
            Update $\Sigma^{t+1} := \arg\min \; \norm{\Sigma (V^t)^T Y -  (U^{t+1})^T X^t}^2_F \;\; \text{s.t.} \;\;  \kappa(\Sigma)\le \rho$ by solving the 1D convex optimization problem \eqref{projection_problem} \;
            Update singular spectrum to obey Frobenius norm constraint $\Sigma^{t+1} := \tau \Sigma^{t+1} / \| \Sigma^{t+1} \|_F $ \;
            Use $V^{t+1} := \arg\min  \; \norm{Y^T V  - (X^t)^T U^{t+1} (\Sigma^{t+1})^{-1} }^2_F \;\; \text{s.t.} \;\;  V^TV = I_n$ and solve for $V^{t+1}$ by the closed-form formula for the polar decomposition of matrix $Y (X^{t})^T U^{t+1} (\Sigma^{t+1})^{-1} $ \;
	         Compute $W^{t+1} = U^{t+1} \Sigma^{t+1} (V^{t+1})^T$ \;
         Update sparse codes $X^{t+1} := H_s(W^{t+1}Y)$.
  }
\end{algorithm}

\noindent The AM algorithm alternates between solving minimization over the transform $W$ (lines 3--7) and the codes $X$ (line 8). As previously observed, minimizing over the $\ell_0$ sparsity set reduces to the computation of the hard-thresholding operator (see line 8). We further discuss the computation of the transform having the codes from the previous step.

\noindent The initialization of the decomposition for $W$ can be done by using the least-squares or Procrustes solutions to minimize the quantity $\| X - WY \|_F$. The first produces the unconstrained condition number solution while the latter reaches the perfectly conditioned solution (condition number one), i.e., an orthogonal $W$.


\begin{remark}
The steps of the proposed algorithm which update $U,\Sigma$, and $X$ solve exactly each sub-problem of the original problem and therefore guarantee that the objective function value decreases at each step. Unfortunately, the update of $V$ cannot be solved exactly (to the best of our knowledge), and therefore in this case an increase in the objective function value is possible. As such the proposed algorithm is not monotonically convergent, in general.
\end{remark}

\section{Projection for the constrained singular spectrum} \label{sec:spectrum_projection}

\noindent Given $\kappa \ge 1$ and $d,r \in \rset^n_+$, we consider the following problem of interest:
\begin{align*}
\min_{s} \quad \; & \norm{s - d}^2_2 \\
\text{subject to} \; & \frac{ \max_i {s_i r_i} }{\min_i {s_i r_i}}  \le  \kappa.
\end{align*}
In a nutshell, the problem seeks the best approximation of vector $d$ with a weighted $\kappa-$conditioned vector $s^*$. Trivially, when $\frac{ \max_i {d_ir_i} }{\min_i {d_ir_i}}  \le  \kappa$ then $s^* = d$. Also, if $\kappa = 1$ then $s^*$ is a multiple of the ones vector and the solution is that of the Procrustes optimization problem \cite{schonemann1966generalized}, i.e., $W$ is orthogonal.
An equivalent form of this problem is  \eqref{projection_problem}, restated here for simplicity:
\begin{align}\label{kappa_cond_problem_lowup}
\min_{s,l} \quad \; & \norm{s- d}^2_2 \\
\text{subject to} \; & lr_i \le s_i \le lr_i \kappa \quad \forall i \in \{1, \cdots, n\}.  \nonumber 
\end{align}


\noindent Particularly, if $r_i = 0$ then $s_i$ is constrained to values $0$. However, since \eqref{kappa_cond_problem_lowup} is not an usual projection problem onto a given convex set, an efficient algorithm is not directly visible under this formulation.
The next auxiliary result reformulates  problem \eqref{kappa_cond_problem_lowup} in a simpler form, that facilitates the application of finite direct procedures to solve it. Denote $r_{\min} = \min_i r_i$.

\begin{theorem}\label{lemma_aux}
Let $d \neq 0$ and assume $\frac{ \max_i {d_ir_i} }{\min_i {d_ir_i}}  >  \kappa > 1$. Define $A(l):=\{i: d_i \le lr_i \}, B(l):=\{i: d_i \ge \kappa r_i l\}$. The following relations hold:
\begin{itemize}
\item[$(i)$] The minimization problem \eqref{kappa_cond_problem_lowup} reduces to the following convex piecewise-quadratic minimization:
\begin{align*}
\min_l \; g(l): = \sum_{i \in A(l)} (d_i-r_il)^2 + \sum_{i \in B(l)} (d_i - \kappa r_i l)^2.
\end{align*}

\item[$(ii)$] The extremal points of function $g$ satisfies: 
\begin{align}
   l^* = \frac{\sum_{i \in A(l^*)} r_i d_i + \kappa \sum_{i \in B(l^*)} r_i d_i }{   \sum_{i \in A(l^*)} d_i^2 + \kappa^2 \sum_{i \in B(l^*)} d_i^2}
   \label{eq:loptimum}
\end{align}
 and $    g(l^\star) = \sum\limits_{i \in A(l^\star)} d_i^2 + \sum\limits_{j \in B(l^\star)} d_j^2 - \frac{ \left(\sum_{i \in A(l^*)} r_i d_i + \kappa \sum_{i \in B(l^*)} r_i d_i \right)^2 }{   \sum_{i \in A(l^*)} d_i^2 + \kappa^2 \sum_{i \in B(l^*)} d_i^2}.$
\item[$(iii)$] The function $g$ has a quadratic growth, i.e. the following inequality holds
\begin{align}\label{quadratic_growth}
g(l) - g(l^*) \ge \frac{1}{2}\min\left\{ \sum_{i \in A(l^*)} r_i^2 + \kappa^2 r_{\min}^2, r_{\min}^2 + \kappa^2 \sum_{i \in B(l^*)} r_i^2 \right\}(l-l^*)^2.
\end{align}
\item[$(iv)$] The function $g$ has a unique minimum $l$.
\end{itemize}
\end{theorem}

\begin{proof}
First observe that, for a fixed $l$, the solution of \eqref{kappa_cond_problem_lowup} in $s$ is the projection of $d$ onto the lower-upper bounds $[lr,lr\kappa]$. Thus, for any $i$ we have:
\begin{align}\label{explicit_s}
s_i = 
\begin{cases}
lr_i,  & i \in A(l) \\ 
d_i,  & i \notin A(l) \cup B(l)  \\ 
\kappa lr_i,  & i \in B(l). 
\end{cases}
\end{align}
By replacing \eqref{explicit_s} into the objective of \eqref{kappa_cond_problem_lowup} we obtain $(i)$.

\noindent 
Notice that $g'(l) = 2\sum_{i \in A(l)} r_i(r_il-d_i) + 2\kappa \sum_{i \in B(l)} r_i (\kappa l r_i - d_i) = 2r^T[lr - d]_+ - 2\kappa^2 r^T \left[\frac{d}{\kappa} - lr \right]_+$. The Fermat Theorem implies straightforwardly result $(ii)$. 

\vspace{5pt}

\noindent For simplicity denote $(d/r)_{\max}:= \max_i {\frac{d_i}{r_i}}, (d/r)_{\min}:= \min_i {\frac{d_i}{r_i}}$. Due to our assumption $\frac{ (d/r)_{\max} }{(d/r)_{\min}}  >  \kappa$ it is obvious that $B((d/r)_{\min})$ and $A((d/r)_{\max}/\kappa)$ are nonempty. By observing $$g'((d/r)_{\min}) = 2\kappa \sum_{i \in B((d/r)_{\min})} r_i[\kappa (d/r)_{\min} r_i - d_i] < 0$$ 
and 
$$g'((d/r)_{\max}/\kappa) = 2\sum_{i \in A ((d/r)_{\max}/\kappa)} r_i \left[r_i \frac{(d/r)_{\max}}{\kappa} - d_i \right] > 0,$$ 
we conclude that the extremal points of $g$ lie in $((d/r)_{\min}, (d/r)_{\max}/\kappa)$. By taking into account the explicit form of $g'(\cdot)$ we further derive:
\begin{align}\label{prelim_lowbound}
& [g'(l_1) - g'(l_2)](l_1-l_2) \nonumber\\
& = 2\left \langle [l_1 r - d]_+ - [l_2 r - d]_+ + \kappa^2 \left[\frac{d}{\kappa} - l_2 r \right]_+ - \kappa^2\left[\frac{d}{\kappa} - l_1 r \right]_+, l_1 r - l_2 r  \right \rangle.
\end{align}
Now,  observe on one hand that $A(l_1)$ and $A(l_2)$ are the nonzero support sets of $ [l_1 r - d]_+$ and $[l_2 r - d]_+ $, respectively. This implies a first lower bound:
\begin{align}\label{lowbound1}
\left \langle [l_1 r - d]_+ - [l_2 r - d]_+, l_1 r - l_2 r  \right \rangle \ge  \min \left\{ \sum\limits_{i \in A(l_1)} r_i^2, \sum\limits_{i \in A(l_2)} r_i^2 \right\} (l_1-l_2)^2.
\end{align}
On the other hand, since $B(l_1)$ and $B(l_2)$ are the nonzero support sets of $ \left[\frac{d}{\kappa} - l_1 r \right]_+$ and $ \left[\frac{d}{\kappa} - l_2 r \right]_+$, respectively, a similar bound holds: 
\begin{align}\label{lowbound2}
\left \langle \left[\frac{d}{\kappa} - l_2 r \right]_+ - \left[\frac{d}{\kappa} - l_1 r \right]_+,l_1 r - l_2 r  \right \rangle \ge \min \left\{ \sum\limits_{i \in B(l_1)} r_i^2, \sum\limits_{i \in B(l_2)} r_i^2 \right\} (l_1-l_2)^2.
\end{align}
By substituting \eqref{lowbound1}-\eqref{lowbound2} in \eqref{prelim_lowbound} and by taking $l_2 = l^*$, we obtain the following strong monotonicity property:
\begin{align}\label{prelim_qg}
g'(l)(l-l^*) \ge  2\left( \min \left\{ \sum\limits_{i \in A(l)} r_i^2, \sum\limits_{i \in A(l^*)} r_i^2 \right\} + \kappa^2 \min \left\{ \sum\limits_{i \in B(l)} r_i^2, \sum\limits_{i \in B(l^*)} r_i^2 \right\} \right) (l-l^*)^2, \end{align}
for $l \in ((dr)_{\min}, (dr)_{\max}/\kappa)$.
Since $g$ is differentiable, the Mean Value Theorem (see \cite[Sec. 1.1.1, pg. 4]{Pol:87book}) is applicable and we further reach:
\begin{align*}
& g(l)  = g(l^*) + \int_0^1 \frac{1}{\tau} {g'(l^* + \tau(l-l^*)) \tau(l-l^*)} d\tau \\
& \overset{\eqref{prelim_qg}}{\ge} g(l^*) +  \int_0^1 \frac{2}{\tau} \left( \min \Bigg\{ \sum\limits_{i \in A(l^* + \tau(l-l^*))} r_i^2, \sum\limits_{i \in A(l^*)} r_i^2 \right\} \\
& \hspace{3cm} + \kappa^2 \min \left\{ \sum\limits_{i \in B(l^* + \tau(l-l^*))} r_i^2, \sum\limits_{i \in B(l^*)} r_i^2 \right\} \Bigg) \tau^2(l-l^*)^2   d\tau \\
& \ge g(l^*) + (l-l^*)^2 \min_{t \in [l,l^*]}  \min \left\{ \sum\limits_{i \in A(t)} r_i^2, \sum\limits_{i \in A(l^*)} r_i^2 \right\} + \kappa^2 \min \left\{ \sum\limits_{i \in B(t)} r_i^2, \sum\limits_{i \in B(l^*)} r_i^2 \right\}.
\end{align*}
Finally, notice that for $l_1 \le l_2$ then $A(l_1) \subseteq A(l_2)$ and $B(l_2) \subseteq B(l_2)$. This last observation allows the final bound:
\begin{align*}
\min_{t \in [l,l^*]}  \min \left\{ \sum\limits_{i \in A(t)} r_i^2, \sum\limits_{i \in A(l^*)} r_i^2 \right\} &+ \kappa^2 \min \left\{ \sum\limits_{i \in B(t)} r_i^2, \sum\limits_{i \in B(l^*)} r_i^2 \right\} \\
& \ge 
\begin{cases}
\sum\limits_{i \in A(t)} r_i^2 + \kappa^2 \sum\limits_{i \in B(l^*)} r_i^2, & \text{if} \; t < l^* \\
\sum\limits_{i \in A(l^*)} r_i^2 + \kappa^2 \sum\limits_{i \in B(t)} r_i^2 , & \text{if} \; t \ge l^* \\
\end{cases} \\
& \ge \min\left\{ \sum_{i \in A(l^*)} r_i^2 + \kappa^2 r_{\min}^2, r_{\min}^2 + \kappa^2 \sum_{i \in B(l^*)} r_i^2 \right\},
\end{align*}
we recover result $(iii)$.
\noindent The uniqueness statement of the fourth part $(iv)$ immediately yields from $(iii)$. Assume $l^*$ is not unique, i.e. there exists $l_1^*$ and $l_2^*$ such that $g(l_1^*) = g(l_2^*) = g^*$. By replacing into inequality of $(iii)$ we get:
\begin{align*}
0 \ge \min\left\{ \sum_{i \in A(l^*_1)} r_i^2 + \kappa^2 r_{\min}^2, r_{\min}^2 + \kappa^2 \sum_{i \in B(l^*_1)} r_i^2 \right\} (l_1^*-l_2^*)^2 > 0,
\end{align*}
which leads to a contradiction.
\end{proof}

\begin{remark}
Note that Theorem \ref{lemma_aux} reduces the minimization problem \eqref{kappa_cond_problem_lowup} to a smooth convex one-dimensional search. The main difficulty in the evaluation of $g(l)$ and $g'(l)$ is carried by the explicit computation of the index sets $\{A(l),B(l)\}$, at a cost of at most $\BigO(n^3)$, the cost for each of the two other steps of the proposed algorithm.

Note that the optimization problem in \eqref{kappa_cond_problem_lowup} can be equivalently restated as
\begin{align}\label{kappa_cond_problem_lowup_new}
\min_{s,l} \quad \; & \norm{ d \odot( s- r \oslash d  ) }^2_2 \\
\text{subject to} \; & l \le s_i \le l \kappa \quad \forall i \in \{1, \cdots, n\},  \nonumber
\end{align}
where we have defined $\odot$ as the elementwise multiplication and $\oslash$ is the elementwise division. We would like to point out several simplifications and properties. First, we assume $d_i > 0$ as this is a norm quantity, and if zero then the corresponding term in the objective function can be removed from the optimization problem. We can also assume that $r_i \geq 0$ because we can always change the signs of rows from $X$ in order to guarantee non-negative $r_i$ for all $i$. Second, note that $r \oslash d$ is the solution to the unconstrained optimization problem (unconstrained by the condition number $\kappa$ bound). Without loss of generality, we can assume that the non-negative elements in $r$ are ordered decreasingly (by reordering the rows in the matrices $X$ and $Y$).

The goal is to find the solution $s$ closest to $r \oslash d$ in the sense that as few entries of $r \oslash d$ as possible are modified in the optimization problem to obey the conditioning constraints in terms of $l$ and $\kappa$. If the entries of $r \oslash d$ are ordered, as assumed here, then the problem reduces to finding the sets $A(l)$ and $B(l)$ of the minimum size that obey the constraints. The obvious approach of searching among all possible pairs for the sets $A(l)$ and $B(l)$ leads to an algorithm whose complexity is $\BigO(n^2)$. This is the approach that we use in this paper.

We would like to note that written in this way and with the assumptions that we made, our optimization problem \eqref{kappa_cond_problem_lowup_new} is equivalent to the spectrum calculation problem as defined in \cite{LI2020190} and their proposed algorithm (Algorithm 2.1 of \cite{LI2020190}) might be adapted for our particular scenario to retrieve the best solution $s^\star$. The significant difference with the work \cite{LI2020190} is the new formula \eqref{eq:loptimum} for the calculation of $l^\star$ which minimizes the proposed objective function and the objective function itself which can be viewed as a weighted version of the objective function considered in \cite{LI2020190}. The complexity of this algorithm is $\BigO(n)$ but it is not clear how it can be extended to our approach. 

We would like to note that as the updates for $U$ and $V$ are both $\BigO(n^3)$ the computational benefit of having the spectrum update in $\BigO(n)$ is valuable but not significant in the overall running time of the proposed algorithm.

\end{remark}

\section{Experimental numerical results}
\label{sec:experiments}

In this section, we depict several numerical simulations on synthetic and real data to show-case our method and compare it to existing work.
Our algorithm implementation and its applications are public and available online~\footnote{\url{https://github.com/pirofti/ConditionedTransformLearning}}.

In all the experiments the goal is to design a transform $W$ that has the same properties (we refer here to the condition number and the Frobenius norm of $W$) as the ones designed with the method from \cite{RavBre:12}. In the plots, we refer to the results obtained via \cite{RavBre:12} as the \texttt{bresler} method. Therefore, in each case, we will choose a sparsity level $T$ and design a transform $W_\text{bresler}$ using the \texttt{bresler} method and then apply our algorithm with both the target condition number and the target Frobenius norm of $W_\text{bresler}$ and measure the representation error. We perform the experiments in this way in order to be able to appropriately compare the transforms (compare representation error achieved for the same condition number and Frobenius norm).

All methods are initialized with the same orthogonal transformation $D \otimes D$, where $D$ is the Discrete Cosine Transform matrix of size $\sqrt{n} \times \sqrt{n}$ and $\otimes$ is the Kronecker product.

\subsection{Numerical results on synthetic data}

We show the performance of the proposed algorithm in terms of the objective function described, as compared to the transform learning approach from \cite{RavBre:12}, on image data. We extract $8 \times 8$ non-overlapping patches, subtract their means, and vectorize them, i.e., $n = 64$, using popular test images such as barbara, peppers, and lena. Our dataset $Y$ consists of a concatenation of all these vectorized patches and therefore the size of the dataset is $m = 12288$. This setup is well known and standard in the image processing literature \cite{EladAharon06_denoising}.

Again following the work in \cite{RavBre:12}, we choose to show two representation errors, the regular and the normalized one, like:
\begin{equation}
    \| X - WY \|_F \text{ and } \| X - WY \|_F \| WY \|_F^{-1}.
\end{equation}

Experimental results for sparsity levels $T \in \{6, 8\}$ are shown in Figures \ref{fig:plotab} and \ref{fig:plotcd}, respectively. The \texttt{bresler} method runs with parameters: for the log-determinant regularization we have the weight $\mu \in \{  2.1\times 10^{-5}, 2.1\times 10^{-6}, 2.1\times 10^{-8}, 10^{-9} \} \times \| WY \|_F^2$ while for the Frobenius norm regularization we have the weight $\rho = \mu$. We chose this range in order to reach a wide range of condition numbers for our four transforms: from a few tens to a few tens of thousands. Experimental results show that the proposed method outperforms the \texttt{bresler} approach in the scenarios we consider. While the convergence of our method is slower in some cases we do ultimately reach lower regular and normalized representation error values. Overall, as expected, better representation error is achieved when the condition number of $W$ is larger (mimicking the behavior of the least-squares solution whose condition number is not constrained) and the allowed sparsity level is higher.

\begin{figure}
\centering
\begin{subfigure}[b]{0.75\textwidth}
  \centering
  \includegraphics[width=\linewidth]{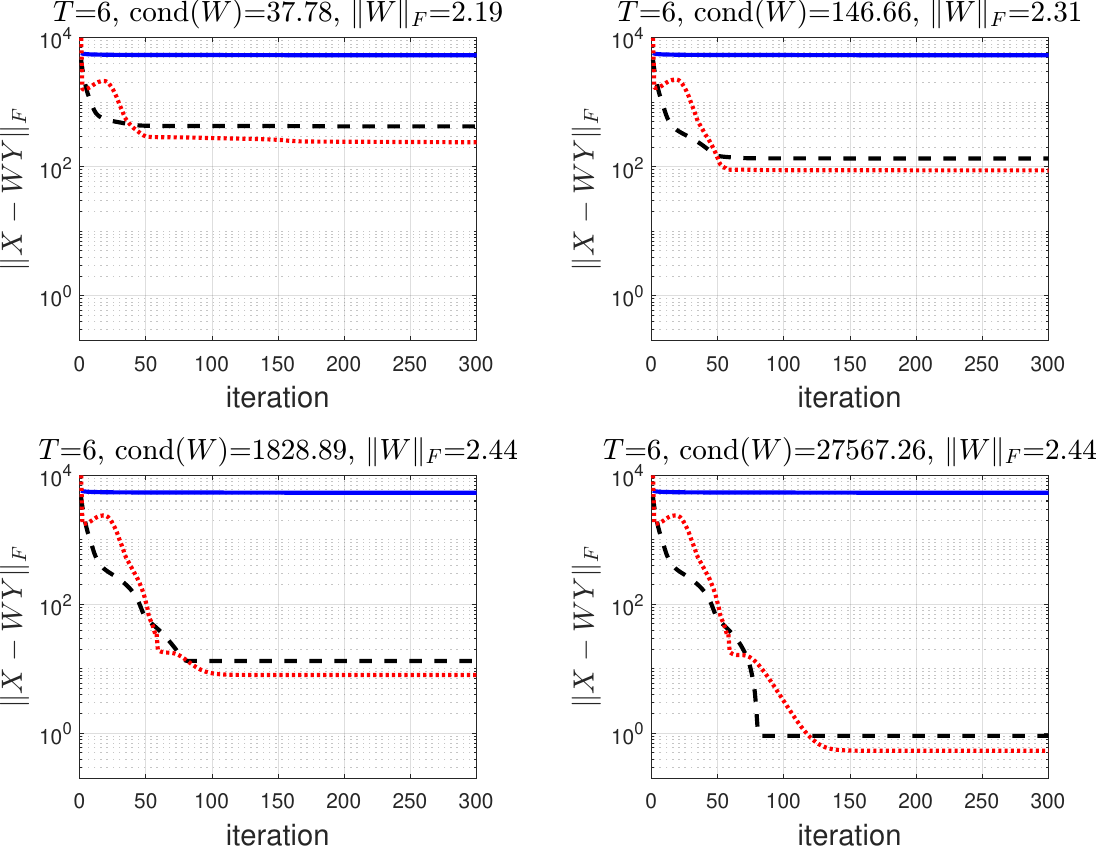}
  \caption{Representation error achieved for $T = 6$ and various condition numbers and Frobenius norms.}
  \label{fig:subab1}
\end{subfigure}
\begin{subfigure}[b]{0.75\textwidth}
  \centering
  \includegraphics[width=\linewidth]{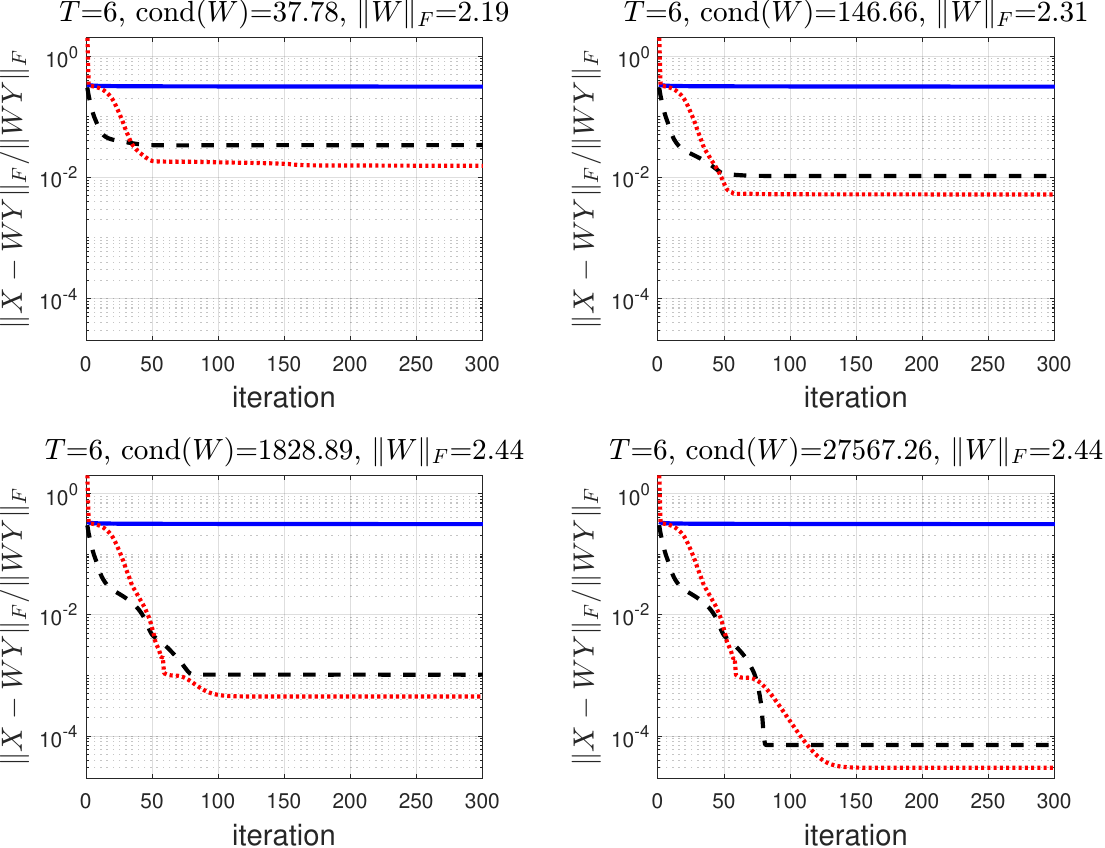}
  \caption{Normalized representation error achieved for $T = 6$ and various condition numbers and Frobenius norms.}
  \label{fig:subab2}
\end{subfigure}
\caption{Evolution of the regular and normalized representation errors for 300 iterations with the sparsity level $T = 6$ for the \texttt{bresler} method (in the dashed black lines) and the proposed method (in the dotted red lines). For perspective, we show the Procrustes solution as well (the continuous blue line).}
\label{fig:plotab}
\end{figure}

\begin{figure}
\centering
\begin{subfigure}[b]{0.75\textwidth}
  \centering
  \includegraphics[width=\linewidth]{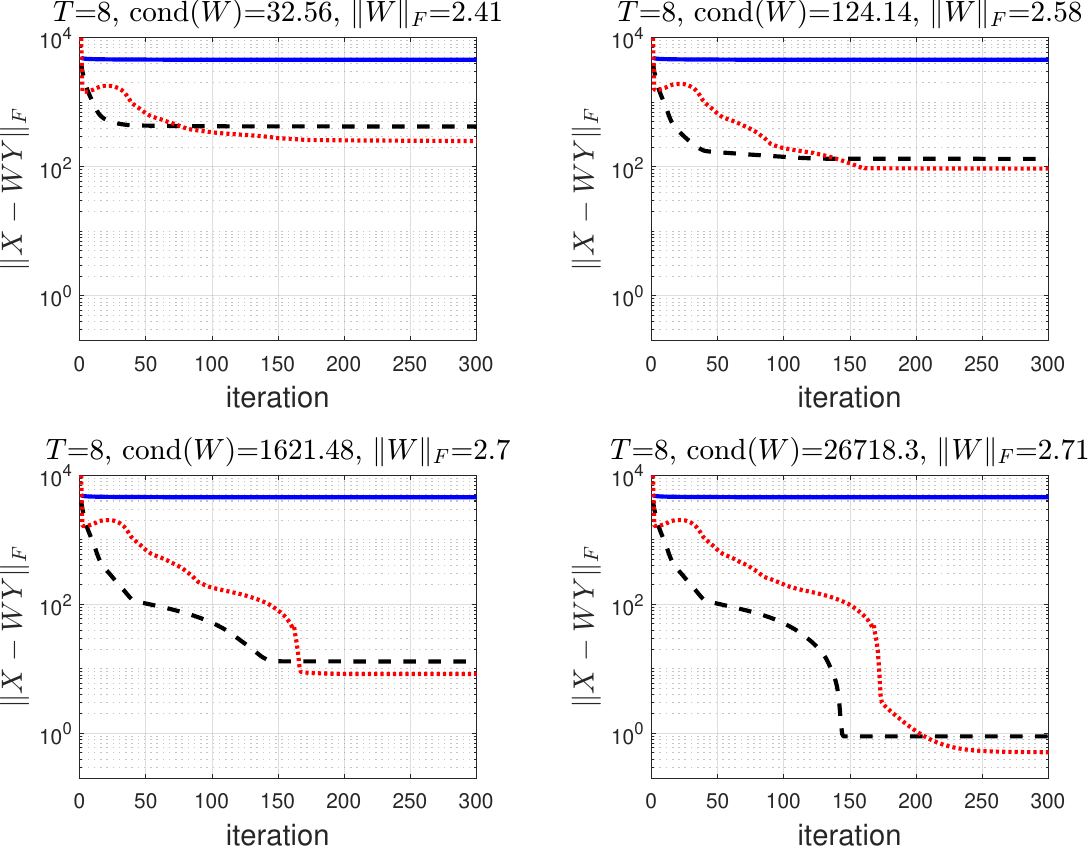}
  \caption{Representation error achieved for $T = 8$.}
  \label{fig:subcd1}
\end{subfigure}
\begin{subfigure}[b]{0.75\textwidth}
  \centering
  \includegraphics[width=\linewidth]{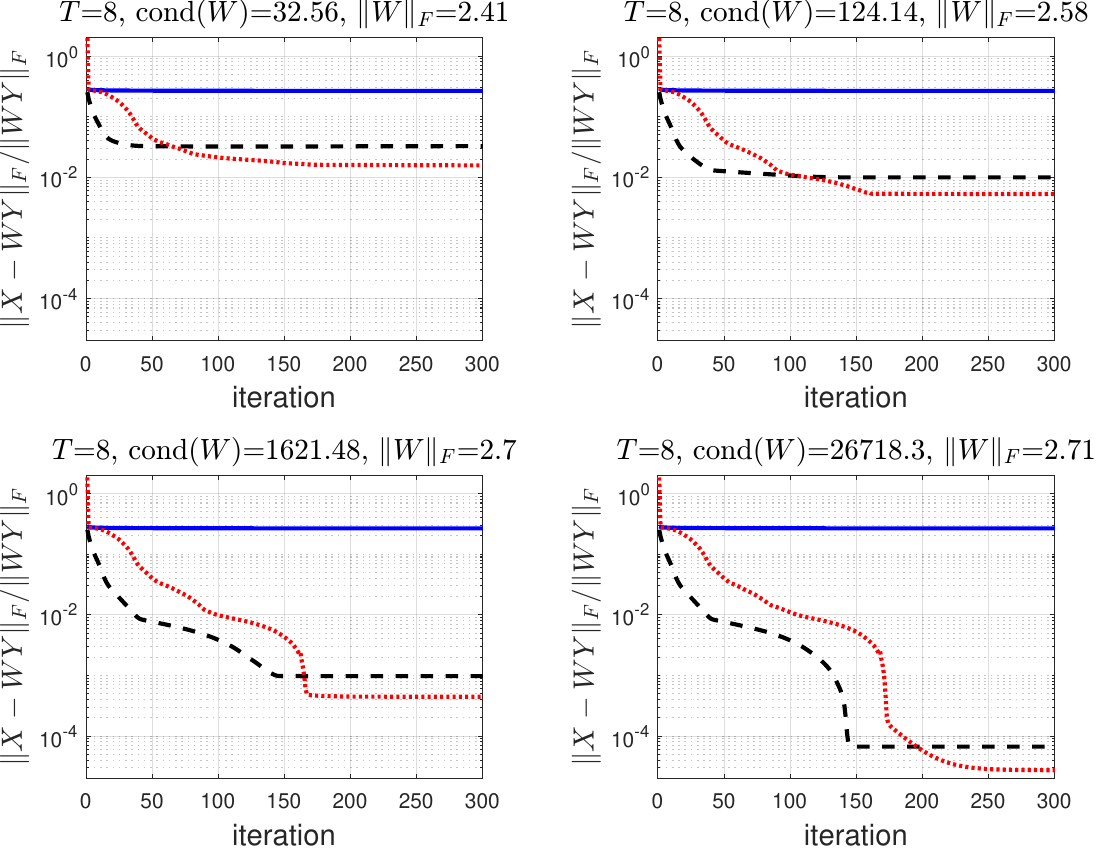}
  \caption{Normalized representation error for $T = 8$.}
  \label{fig:subcd2}
\end{subfigure}
\caption{Analogous to Figure \ref{fig:plotab} but for the sparsity level $T = 8$.}
\label{fig:plotcd}
\end{figure}

\subsection{Numerical results on image data}


Notice that in several denoising and signal processing applications, an optimal conditioning number might be desired to attain the practical transform properties.


We perform denoising tests following the methodology, implementation, and data provided in \cite{RB15_SparseTransoformDenoising}
and additional data from \cite{RB13_DoubleSparseImages}.
We thank the authors for making their code and data public and for encouraging reproducible research.
All clean and noisy images were used as provided by the original authors.
Noisy data was for all images at $\sigma_\text{noise} = \{5, 10, 15, 20, 100\}$ levels.
We took
Barbara, Cameraman, Couple, and Brain from~\cite{RB15_SparseTransoformDenoising}
and 
Hill, Lena, Man from \cite{RB13_DoubleSparseImages}.

Our experiments include three methods:
our proposed method from Algorithm~\ref{alg:rip},
the \texttt{bresler} method from~\cite{RavBre:12} with the setup described in \cite{RB15_SparseTransoformDenoising}
(which partially relies on the setup from \cite{RB13_DoubleSparseImages})
and the Procrustes-based orthogonal transform learning method which we will call \texttt{Ortho} from now on.

Denoising starts by vectorizing overlapping image patches of size $n = p \times p$ from the provided noisy image.
The resulting patches are arranged in matrix $Y\in\rset^{n \times m}$.
This dataset is provided for all tests. Adapted from \cite{RB15_SparseTransoformDenoising}, our denoising transform learning problem is:
\begin{align*}
\min\limits_{W,X,\hat{Y}} \quad \; & \norm{W \hat{Y}-X}^2_F + \beta\norm{Y-\hat{Y}}^2_F \\
\text{subject to} \;\; & \norm{W}_F = \tau, \; \kappa(W)\le  \rho, \; \norm{X_i}_0 \le s \quad \forall i,  \nonumber
\end{align*}
where $\hat{Y}$ is the cleaned signal set.
The objective follows that of $\eqref{approx_relaxed_constnorm_model}$
where the regularization term was added
to control the distance between the noisy and clean
matrices through the $\beta$ parameter.
We solve this through alternating optimization.
First, we compute the optimal transform $W^*$,
where $X$ and $\hat{Y}$ are fixed,
with respect to the particular formulation of each algorithm:
\eqref{approx_relaxed_constnorm_model} for the proposed scheme,
Procrustes' $W=UV^T$ where $YX^T = U\Sigma V^T$,
and, respectively, by (P3) from \cite{RB15_SparseTransoformDenoising} for \texttt{bresler}.
Then, using the obtained optimal transform, all the algorithms compute optimal $X^*$ through hard-thresholding the columns of $W^*\hat{Y}$.
In our experiments, we follow the much more versatile variable sparsity update strategy described when solving (P4) in \cite{RB13_DoubleSparseImages},
such that representation is performed in a loop where the sparsity is increased until a given threshold error is achieved $\varepsilon = C \sqrt{n} \sigma_\text{noise}$.
Finally $\hat{Y}$ is obtained by keeping $W^*$ and $X^*$
fixed and solving the resulting least-squares problem.
We repeat these steps until convergence is obtained.
After the learning iterations, the final step performs pixel averaging from the overlapping patches to restore a smooth cleaned image.
These steps are presented in great detail in \cite{RB15_SparseTransoformDenoising},
here we follow them exactly and just replace the transform learning method where necessary.

For the \texttt{bresler} method we use the same parametrization as the one recommended in the implementation and Table~I from \cite{RB15_SparseTransoformDenoising}:
we use $11\times 11$ patches with a stride of 1 pixel, with $I=20$ learning iterations and $T=12$ inner transform learning iterations.
We
set $C=1.15$ as per \cite{EladAharon06_denoising}. 
All methods are initialized with the same data and parameters values where applicable.
As earlier described,
our method has extra constraints based on the results obtained by the \texttt{bresler} method.

Table~\ref{tab:denoising_p11} depicts the obtained results.
We compare denosining quality via peak signal-to-noise ratio (PSNR) and the structural similarity index (SSIM)~\cite{WBSS04_ssim}
between the original and cleaned images.
The first column represents the noise levels and the 
PSNR between the original and noisy images.
For the reader's convenience, we mark the best results in each image experiment per each quality measure.
While the proposed method performs better most times,
and always better than \texttt{bresler} method,
it comes as a surprise how well the Ortho method performs in the denoising scenario.
Because the choice of $p=8$ patches is
often found in the literature
we performed an identical set of tests as above with only this parameter change and arrived at similar results as those in Table~\ref{tab:denoising_p11}.

\begin{table*}
  
\tabcolsep 2.7pt
\caption{Denoising PSNR(dB) and SSIM for standard images ($p=121$, $C=1.15$)}
\label{tab:denoising_p11}
\footnotesize
\bc \bt{c c c c c c c c c c c c c c c c}

\hline
\multirow{2}{*}{\shortstack[l]{$\sigma_\text{noise}$ \\/ PSNR}} &
\multirow{2}{*}{Method}

&\multicolumn{2}{c}{barbara}
&\multicolumn{2}{c}{cameraman}
&\multicolumn{2}{c}{couple}
&\multicolumn{2}{c}{hill}
&\multicolumn{2}{c}{lena}
&\multicolumn{2}{c}{man}
&\multicolumn{2}{c}{brain}
\\

& &
PSNR & SSIM &   
PSNR & SSIM &
PSNR & SSIM &
PSNR & SSIM &
PSNR & SSIM &
PSNR & SSIM &
PSNR & SSIM \\
\hline

\multirow{3}{*}{\shortstack[c]{$5$ /\\ $34.15$}}
& Proposed & 37.63 & \textbf{0.987} & 37.31 & 0.953 & 36.91 & 0.984 & \textbf{36.73} & \textbf{0.982} & 38.05 & 0.981 & \textbf{36.37} & 0.992 & \textbf{42.00} & 0.985 \\
& Bressler & 37.67 & 0.986 & \textbf{37.32} & 0.953 & 36.90 & 0.984 & 36.64 & 0.981 & 38.05 & 0.981 & 36.36 & 0.992 & 41.88 & 0.985 \\
& Ortho & \textbf{37.68} & 0.986 & \textbf{37.32} & 0.953 & \textbf{36.93} & 0.984 & 36.67 & 0.981 & \textbf{38.07} & 0.981 & \textbf{36.37} & 0.992 & 41.96 & 0.985 \\
\hline
\multirow{3}{*}{\shortstack[c]{$10$ /\\ $28.13$}}
& Proposed & 33.20 & 0.966 & 33.01 & 0.915 & 32.67 & \textbf{0.953} & \textbf{32.73} & \textbf{0.944} & \textbf{34.70} & \textbf{0.960} & \textbf{32.32} & \textbf{0.975} & \textbf{37.57} & 0.960 \\
& Bressler & 33.35 & 0.967 & 33.01 & 0.915 & 32.63 & 0.951 & 32.45 & 0.938 & 34.47 & 0.958 & 32.16 & 0.974 & 37.32 & 0.960 \\
& Ortho & \textbf{33.49} & \textbf{0.968} & \textbf{33.04} & \textbf{0.916} & \textbf{32.75} & \textbf{0.953} & 32.64 & 0.941 & 34.55 & 0.958 & 32.22 & 0.975 & 37.48 & 0.960 \\
\hline
\multirow{3}{*}{\shortstack[c]{$15$ /\\ $24.58$}}
& Proposed & \textbf{30.70} & 0.943 & \textbf{30.58} & \textbf{0.877} & \textbf{30.44} & \textbf{0.917} & 30.34 & 0.892 & \textbf{32.57} & \textbf{0.936} & \textbf{30.16} & \textbf{0.954} & \textbf{34.98} & 0.927 \\
& Bressler & 30.66 & 0.943 & 30.44 & 0.875 & 30.23 & 0.911 & 30.11 & 0.884 & 32.36 & 0.934 & 29.80 & 0.948 & 34.45 & 0.927 \\
& Ortho & 30.65 & 0.943 & 30.55 & \textbf{0.877} & 30.40 & 0.915 & \textbf{30.44} & \textbf{0.894} & 32.48 & 0.935 & 29.95 & 0.952 & 34.89 & \textbf{0.929} \\
\hline
\multirow{3}{*}{\shortstack[c]{$20$ /\\ $22.12$}}
& Proposed & \textbf{28.90} & \textbf{0.915} & \textbf{28.82} & \textbf{0.840} & 28.77 & \textbf{0.879} & 28.90 & \textbf{0.844} & \textbf{31.03} & \textbf{0.912} & 28.44 & \textbf{0.928} & \textbf{32.93} & \textbf{0.901} \\
& Bressler & 28.73 & 0.913 & 28.55 & 0.835 & 28.43 & 0.869 & 28.63 & 0.830 & 30.80 & 0.908 & 28.10 & 0.916 & 32.38 & 0.897 \\
& Ortho & 28.79 & 0.914 & 28.67 & 0.838 & \textbf{28.80} & 0.878 & \textbf{28.91} & 0.843 & \textbf{31.03} & \textbf{0.912} & \textbf{28.46} & 0.925 & 32.91 & \textbf{0.901} \\
\hline
\multirow{3}{*}{\shortstack[c]{$100$ /\\ $8.11$}}
& Proposed & \textbf{21.35} & \textbf{0.616} & \textbf{20.76} & \textbf{0.603} & \textbf{21.99} & \textbf{0.539} & 23.64 & 0.583 & \textbf{23.94} & \textbf{0.695} & \textbf{22.26} & \textbf{0.670} & \textbf{23.98} & \textbf{0.536} \\
& Bressler & 21.20 & 0.608 & 20.20 & 0.581 & 21.91 & 0.535 & \textbf{23.65} & 0.583 & 23.69 & 0.686 & 22.07 & 0.663 & 23.63 & 0.526 \\
& Ortho & 21.22 & 0.609 & 20.47 & 0.592 & 21.92 & 0.535 & \textbf{23.65} & 0.583 & 23.72 & 0.687 & 22.09 & 0.663 & 23.65 & 0.527 \\
\hline

\et \ec
\end{table*}

  






\section{Conclusions}
\label{sec:conclusion}

In this paper, we proposed a new alternating optimization algorithm to learn well-conditioned transformations for sparse representations. The novelty of the proposed method lies in the way we can explicitly upper bound the condition number in all the steps of the method. Each step of the proposed algorithm is solved by numerically efficient procedures based on Procrustes and constrained convex one-dimensional minimization problems. Numerical results show that we outperform previous methods from the literature on synthetic and image-denoising scenarios.

In the future, we plan to analyze and adapt the proposed algorithm for
the large data setting,
where the polar decomposition and singular value decomposition computational costs become prohibitive,
and also for the online formulation,
where the full dataset is provided
sequentially in mini-batches
and the transform has to be updated accordingly.

\bibliographystyle{plain}
\bibliography{bcd}

\end{document}